\documentclass[a4paper,leqno]{amsart}

\usepackage{amsmath,amssymb,amsfonts,amsthm,mathrsfs}

\usepackage{hyperref}
\usepackage{pdfsync}

\newcommand{\C}{\mathbb C}

\newcommand{\R}{\mathbb R}
\newcommand{\N}{\mathbb N}

\newcommand{\Z}{\mathbb Z}

\def\({\left(}
\def\){\right)}
\def\<{\left\langle}
\def\>{\right\rangle}

\def\d{{\partial}}
\def\eps{\varepsilon}

\def\le{\leqslant}
\def\ge{\geqslant}

\DeclareMathOperator{\RE}{Re}
\DeclareMathOperator{\IM}{Im}

\def\Tend#1#2{\mathop{\longrightarrow}\limits_{#1\rightarrow#2}}

\theoremstyle{plain}
\newtheorem{theorem}{Theorem} [section]
\newtheorem{lemma}[theorem]{Lemma}
\newtheorem{corollary}[theorem]{Corollary}
\newtheorem{proposition}[theorem]{Proposition}

\theoremstyle{remark}
\newtheorem{remark}[theorem]{Remark}

\theoremstyle{definition}

\def\Tend#1#2{\mathop{\longrightarrow}\limits_{#1\rightarrow#2}}

\numberwithin{equation}{section}

\begin{document}

\title[Vortex states in NLS with LHY correction]{Self-Bound vortex states in nonlinear Schr\"odinger equations with LHY correction}  

\author[A. K. Arora]{Anudeep K. Arora}
\address{Department of Mathematics, Statistics, and
	Computer Science\\
	University of Illinois at Chicago\\
	851 S. Morgan Street
	Chicago\\ IL 60607, USA}
\email{anudeep@uic.edu}

\author[C. Sparber]{Christof Sparber}
\address{Department of Mathematics, Statistics, and
  Computer Science\\
  M/C 249\\
University of Illinois at Chicago\\
851 S. Morgan Street
Chicago\\ IL 60607, USA} 
\email{sparber@uic.edu}

\begin{abstract}
We study the cubic-quartic nonlinear Schr\"odinger equation (NLS) in two and three spatial dimension. This equation arises in 
the mean-field description of Bose-Einstein condensates with Lee-Huang-Yang correction. We first prove 
global existence of solutions in natural energy spaces which allow for the description of self-bound 
quantum droplets with vorticity. Existence of such droplets, described as central vortex states in 2D and 3D, 
is then proved using an approach via constrained energy minimizers. 
A natural connection to the NLS with repulsive inverse-square potential in 2D arises, leading 
to an orbital stability result under the corresponding flow.
\end{abstract}

\date{\today}


\thanks{This publication is supported by the MPS Simons foundation through award no. 851720}
\maketitle


\section{Introduction}\label{sec:intro}

In recent physics experiments on Bose-Einstein condensates, a novel type of self-bound quantum state, referred to as {\it quantum droplets}, has been formed with the help of 
quantum fluctuations derived from collective Bogoliugov-modes, see \cite{CTSN, LZLYM2021}. These fluctuations can then be described in leading order by the {\it Lee-Huang-Yang} (LHY) 
correction \cite{LHY}, whose rigorous derivation has been discussed in, e.g.,  \cite{Brietzke2020, Fournais2019, FoSo2020}. 
Of particular interest is the fact that such quantum droplets can carry intrinsic {\it vorticity}, a feature which is only possible in 
spatial dimensions two and higher, cf. \cite{CSHM, LCLHTPW, LXCYML}. 
Higher dimensional self-bound quantum states (or solitary wave solutions), are usually strongly unstable due to the possibility of finite-time blow-up, see, e.g., \cite{CazCourant}. The 
inclusion of the LHY correction is a remedy against this unwanted effect and consequently allows for the description of stable quantum droplets. 
In this paper, we shall discuss a possible mathematical description of such solitary waves with non-zero vorticity in dimensions $d=2,3$.

On the level of the associated mean-field model, 
Bose-Einstein condensates are described by a nonlinear Schr\"odinger equation (NLS), governing the dynamics
condensate wave function $u$. In the case with LHY correction this yields an NLS model with 
competing {\it cubic-quartic nonlinearity}:
\begin{equation}\label{eq:nls}
\left\{
\begin{aligned}
 & i\d_t u +\frac{1}{2}\Delta u = - |u|^2u +  |u|^3u,\quad  t \in \R,\ x\in \R^d,\\
  & u_{\mid t=0}  = u_0(x),
\end{aligned}  
\right.
\end{equation}
where $d=2,3$.
Here, the cubic term describes the usual attractive mean-field interaction while the LHY correction yields a repulsive quartic term. Note that, 
during their lifespan, solutions to \eqref{eq:nls} formally conserve the {\it total energy}  
\begin{equation}\label{eq:energy}
  E(u) := \frac{1}{2}\|\nabla u\|_{L^2(\R^d)}^2 -\frac{1}{2}\|u\|_{L^4(\R^d)}^4 + \frac{2}{5} \|u\|_{L^5(\R^d)}^5,
\end{equation}
along with their {\it mass}, {\it momentum}, and {\it angular momentum}, i.e.,
\begin{equation}\label{eq:mass}
  M(u):=\int_{\R^d}|u|^2\, dx,\quad P(u) := \int_{\R^d}\IM \overline u \nabla u\, dx,\quad L(u) = \langle u, L_z u \rangle_{L^2}.
\end{equation}
Here, we have assumed w.r.o.g. that rotations occurs around the $z$-axis, in which case the angular momentum operator $L_z$ is given by 
\begin{equation}\label{eq:angular}
L_zu:= i (x_1 \partial_{x_2} u -x_2 \partial_{x_1} u).
\end{equation}
Note that in polar coordinates $(x_1, x_2)\mapsto (r,\theta)\in \R^2$ 
and in cylindrical coordinates $(x_1, x_2, z)\mapsto (r,\theta, z)\in \R^3$, we simply have \[L_z u = i \partial_\theta u.\] 
We say that a quantum state $u\in L^2(\R^d)$ {\it has non-zero vorticity} whenever $L(u)\not =0$.

The LHY correction is known to be able to arrest the collapse of attractive Bose gases induced by the cubic mean-field force \cite{PhysRevLett.115, PhysRevLett117}. 
Indeed, when combined with H\"older's inequality, the defocusing nature of the quartic nonlinearity 
\begin{equation}\label{eq:L4normHolder}
	\|u\|_{L^4(\R^d)}^4 \le \|u\|_{L^2(\R^d)}^{2/3}\|u\|_{L^5(\R^d)}^{10/3},
\end{equation} 
shows that the total energy is bounded below. This suggests that the occurrence of finite time blow-up due to the focusing cubic part is 
no longer a possibility (a fact we shall rigorously prove below).

This paper is now organized as follows: In Section \ref{sec:Cauchy}, we shall prove global well-posedness of solutions $u$ to the Cauchy problem \eqref{eq:nls} 
in the physical energy spaces $H^1(\R^d)$ and also in the conformal space $\Sigma$, which is needed 
to define the angular momentum $L(u)$. 
In Section \ref{sec:m-vortexstates}, we introduce a class of stationary solution to \eqref{eq:nls} with non-zero vorticity (called central vortex states) which are proportional to 
eigenfunctions of $L_z$. We collect some of their basic properties, in particular their range of admissible frequencies. Existence in $d=2$ is then proved in Section \ref{sec:minimizer2d}, by means of a constrained energy-minimization problem. In Section \ref{sec:inverse}, we shall show that these vortex states 
can also be interpreted as stationary solutions to yet another equation. Namely, a nonlinear Schr\"odinger equation with repulsive inverse-square potential. The fact that they are energy minimizers then allows us to prove orbital stability under the flow with inverse-square potential (in the case of radially symmetric initial data). In turn this also yields 
orbital stability under the original flow in the case without vorticity.
In view of these results, we conjecture that the nature of instability of such central vortex states can 
be described by the effect of a repulsive inverse-square potential. Finally, we shall prove existence of (cylindrically symmetric) 
central vortex states in 3D in Section \ref{sec:minimizer3d}, where we also add some concluding remarks.


\section{Cauchy problem}\label{sec:Cauchy}

In this section, we study the Cauchy problem associated to \eqref{eq:nls}, when rewritten through Duhamel's formula, i.e.,
\begin{equation}\label{eq:duhamel}
u(t) = e^{i\frac{t}{2} \Delta } u_0 - i  \int_0^t e^{i\frac{t-s}{2} \Delta} f(u)(s) \, ds,
\end{equation}
where here and in the following, we denote the nonlinearity by
\[
f(z)= -|z|^2 z + |z|^3 z, \quad z\in \C.
\]
Recall that for three-dimensional Schr\"odinger equation, a {\it Strichartz-pair} is {\it admissible} if
\[
\frac{2}{q} + \frac{3}{r} = \frac{3}{2},\quad 2\le r\le 6,
\]   
and in a two-dimensional setting, the admissibility is given by
\[
\frac{2}{q} + \frac{2}{r} = 1,\quad 2\le r< \infty.
\]
We use the following admissible Strichartz-pairs:
\begin{equation}\label{eq:stri1}
(q_1,r_1) := \left\{
\begin{aligned}
  &\,\,\,\,\left(4,4\right),\quad&\text{if}\,\,\,d=2,\\
   &\left(8,\frac{12}{5}\right),\quad&\text{if}\,\,\,d=3,
\end{aligned}  
\right.
\end{equation}
and
\begin{equation}\label{eq:stri2}
(q_2,r_2) := \left\{
\begin{aligned}
  &\left(\frac{10}{3},5\right),\quad&\text{if}\,\,\,d=2,\\
   &\left(\frac{20}{3},\frac{5}{2}\right),\quad&\text{if}\,\,\,d=3.
\end{aligned}  
\right.
\end{equation}

We start by first proving local well-posedness of \eqref{eq:nls} in the natural energy space. This will be done through a fixed point argument, based on the use of Strichartz estimates. 
This approach is by now classical and follows along the same lines as in \cite{Kato1987} (see also \cite{CazCourant}).
However, due to the physical significance of the LHY correction and 
in order to make the presentation self-contained we shall provide full details on our specific choice of admissible pairs (see below).

\begin{proposition}[Local well-posedness]\label{prop:LWP}
Let $d=2,3$. For any $u_0\in H^1(\R^d)$ and the admissible Strichartz-pairs defined in \eqref{eq:stri1}, \eqref{eq:stri2}, there exist $T>0$ and a unique solution 
\[
u\in \mathcal{X} := C([0, T];H^1(\R^d))\cap L^{q_1}((0, T); W^{1,r_1}(\R^d))\cap L^{q_2}((0, T); W^{1,r_2}(\R^d)) , 
\]
to \eqref{eq:duhamel}, depending continuously on $u_0$, and 
we also have the blow-up alternative, i.e., if $T<\infty$, then
\[
\lim_{t\to T} \| u(t, \cdot)\|_{H^1(\R^d)}=\infty.
\]
\end{proposition}

In view of the fact that \eqref{eq:nls} is time-reversible, we also obtain the analogous statement backward in time.

\begin{proof}
For $T>0$, to be specified later, we define 
\[
\nu(u)=\max\big\{\|u\|_{L^{\infty}(I;H^{1})},\,\,\|u\|_{L^{q_1}(I;W^{1,r_1})},\,\,\|u\|_{L^{q_2}(I;W^{1,r_2})}\big\},
\] 
and for an appropriately defined constant $M>0$, also specified later, let
\begin{align}\label{contractspace}
	B_M=\{u\in \mathcal{X}\,:\nu(u)\le M \}.
\end{align}
We prove that the following operator 
\begin{align}\label{Dope}
	\Phi(u(t))=e^{i\frac{t}{2}\Delta}u_0 - i \int_0^t e^{i\frac{t-s}{2}\Delta}f(u(s)) \,ds
\end{align} 
is a contraction on the set $B_M$. Using Strichartz estimates, we obtain
\begin{equation}
\begin{aligned}\label{Lwp1}
	\|\Phi(u(t))\|_{L^{\infty}(I;L^{2})} &+ \|\Phi(u(t))\|_{L^{q_1}(I;L^{r_1})} + \|\Phi(u(t))\|_{L^{q_2}(I;L^{r_2})} \\
	&\lesssim \| u_0\|_{L^2(\R^d)} + \||u|^2u\|_{L^{q_1^{\prime}}(I;L^{r_1^{\prime}})} + 
	 \||u|^3u\|_{L^{q_2^{\prime}}(I;L^{r_2^{\prime}})},
\end{aligned}
\end{equation}
and
\begin{equation}
\begin{aligned}\label{Lwp2}
	& \|\nabla\Phi(u(t))\|_{L^{\infty}(I;L^{2})}  + \|\nabla\Phi(u(t))\|_{L^{q_1}(I;L^{r_1})} + \|\nabla\Phi(u(t))\|_{L^{q_2}(I;L^{r_2})} \\
	& \, \lesssim \| \nabla u_0\|_{L^2(\R^d)} + \||u|^2 \nabla u\|_{L^{q_1^{\prime}}(I;L^{r_1^{\prime}})} + 
	\||u|^3 \nabla u\|_{L^{q_2^{\prime}}(I;L^{r_2^{\prime}})}.
\end{aligned}
\end{equation}
We perform the computation in 2D by taking $(q_1,r_1)$ and $(q_2,r_2)$ as in \eqref{eq:stri1} and \eqref{eq:stri2}, respectively. A similar argument applies to the 3D case. Using H\"older's and Sobolev's inequality, we get
\begin{align*}
	\||u|^2u&\|_{L^{4/3}(I;L^{4/3})} + 
	 \||u|^3u\|_{L^{10/7}(I;L^{5/4})} \\
	&\lesssim T^{\theta_1}\|\nabla u\|^2_{L^{\infty}(I;L^{2})} \|u\|_{L^{4}(I;L^{4})} 
	+ T^{\theta_2}\|\nabla u\|^3_{L^{\infty}(I;L^{2})} \|u\|_{L^{10/3}(I;L^{5})},\end{align*}
	 as well as
\begin{align*}
	\||u|^2 &\nabla u\|_{L^{4/3}(I;L^{4/3})} + 
	\||u|^3 \nabla u\|_{L^{10/7}(I;L^{5/4})}\\
	&\lesssim T^{\theta_1}\|\nabla u\|^2_{L^{\infty}(I;L^{2})}\|\nabla u\|_{L^{4}(I;L^{4})} 
	+ T^{\theta_2}\|\nabla u\|^3_{L^{\infty}(I;L^{2})}\|\nabla u\|_{L^{10/3}(I;L^{5})},
\end{align*}
where $\theta_1=\frac{1}{2}$ and $\theta_2=\frac{2}{5}$.
Combining \eqref{Lwp1} and \eqref{Lwp2} together with the last two estimates, we obtain
\begin{align*}
	 \|\Phi(u(t))\|_{L^{\infty}(I;H^1)} \: +  &\:\|\Phi(u(t))\|_{L^{4}(I;W^{1,4})} + \|\Phi(u(t))\|_{L^{10/3}(I;W^{1,5})}\\
	 &\lesssim \|u_0\|_{H^1(\R^2)}+T^{\theta_1}\|\nabla u\|^2_{L^{\infty}(I;L^{2})} \|u\|_{L^{4}(I;W^{1,4})} \\
	 &\qquad\qquad\qquad+T^{\theta_2}\|\nabla u\|^3_{L^{\infty}(I;L^{2})} \|u\|_{L^{10/3}(I;W^{1,5})}.	
	 \end{align*}
	 Similarly, for $d=3$, we obtain 
	  \begin{align*}
	\|\Phi(u(t))\|_{L^{\infty}(I;H^1)} \: +  &\:\|\Phi(u(t))\|_{L^{q_1}(I;W^{1,r_1})} + \|\Phi(u(t))\|_{L^{q_2}(I;W^{1,r_2})}\\ 
	 &\lesssim \|u_0\|_{H^1(\R^3)} + T^{\theta_1}\|\nabla u\|^2_{L^{q_1}(I;L^{r_1})} \|u\|_{L^{q_1}(I;W^{1,r_1})}\\
	 &\qquad\qquad\qquad+T^{\theta_2}\|\nabla u\|^3_{L^{q_2}(I;L^{r_2})} \|u\|_{L^{q_2}(I;W^{1,r_2})},
\end{align*}
where $\theta_1=\frac{1}{2}$, $\theta_2=\frac{1}{4}$, $(q_1,r_1)$ and $(q_2,r_2)$ as in \eqref{eq:stri1} and \eqref{eq:stri2}, respectively.
Then, for $u\in B_M$, we have 
\begin{equation}
\begin{aligned}\label{Lwp5}
	& \|\Phi(u(t))\|_{L^{\infty}(I;H^1)} + \|\Phi(u(t))\|_{L^{q_1}(I;W^{1,r_1})} + \|\Phi(u(t))\|_{L^{q_2}(I;W^{1,r_2})} \\
	&\, \le C\|u_0\|_{H^1(\R^d)} + CT^{\theta_1}M^3 + CT^{\theta_2}M^4.
\end{aligned}
\end{equation}
Set $M=2C\|u_0\|_{H^1(\R^d)}$ and take $T>0$ such that
\begin{align}\label{Lwp6}
	CT^{\theta_1}M^{2} + CT^{\theta_2}M^{3} \le \frac{1}{2},
\end{align}
yielding that \eqref{Lwp5} is bounded by $M$. Therefore, for some time $T=T(\|u_0\|_{H^1})>0$, we obtain $\Phi:\,B_M\rightarrow B_M.$ 

To complete the proof we need to show that the operator $\Phi$ is a contraction. This is achieved by running the same argument as above on the difference 
\[
d(\Phi(u(t)),\Phi(v(t))):=\|\Phi(u(t))-\Phi(v(t))\|_{L^{q_1}(I;L^{r_1})} + \|\Phi(u(t))-\Phi(v(t))\|_{L^{q_2}(I;L^{r_2})}
\]
for $u,v\in B_M$. Applying Strichartz estimates, we get
\[
d(\Phi(u(t)),\Phi(v(t))) \lesssim \||u|^2 u - |v|^2 v\|_{L^{q_1^{\prime}}(I;L^{r_1^{\prime}})} + \||u|^3 u - |v|^3 v\|_{L^{q_2^{\prime}}(I;L^{r_1^{\prime}})},
\]
where in $d=2$
\begin{equation}\label{eq:contract1}\begin{aligned}
\||u|^2 u &- |v|^2 v\|_{L^{q_1^{\prime}}(I;L^{r_1^{\prime}})}\\
&\lesssim T^{\theta_1} \Big(\|\nabla u\|^2_{L^{\infty}(I;L^{2})} + \|\nabla v\|^2_{L^{\infty}(I;L^{2})}\Big) \|u-v\|_{L^{q_1}(I;L^{r_1})},
\end{aligned}
\end{equation}
and
\begin{equation}\label{eq:contract2}
\begin{aligned}
\||u|^3 u &- |v|^3 v\|_{L^{q_2^{\prime}}(I;L^{r_1^{\prime}})} \\
&\lesssim T^{\theta_2}
\Big(\|\nabla u\|^3_{L^{\infty}(I;L^{2})} + \|\nabla v\|^3_{L^{\infty}(I;L^{2})}\Big)\|u-v\|_{L^{q_2}(I;L^{r_2})}.
\end{aligned}
\end{equation}
A similar computation in 3D for \eqref{eq:contract1} and \eqref{eq:contract2} yields the pairs $(q_1,r_1)$ and $(q_2,r_2)$, respectively, in place of $(\infty,2)$. 
Thus, for $u,v\in B_M$, we have
\[
d(\Phi(u(t)),\Phi(v(t))) \lesssim \big(T^{\theta_1}M^{2} + T^{\theta_2}M^{3}\big)d(u,v),
\]
which with the smallness condition \eqref{Lwp6} implies that $\Phi$ is a contraction on $B_M$. The continuous dependence follows by a similar argument.
\end{proof}

The space $H^1(\R^d)$ is  the physical energy space associated to \eqref{eq:nls}. In particular, the mass, energy, and momentum are well defined for $u\in H^1(\R^d)$. 
The angular momentum $L(u)$, however, is {\it not} necessarily well defined for $u\in H^1(\R^d)$. 
To overcome this issue, a natural approach is to introduce the conformal space 
\[
\Sigma:=\left\{f\in H^1(\R^d);\ |x |f \in
L^2(\R^d)\right\}, 
\]
equipped with the norm
\[
\|f\|^2_\Sigma =
\|f\|^2_{L^2(\R^d)}+ \left\| \nabla f\right\|^2_{L^2(\R^d)}+\left\| x f\right\|^2_{L^2(\R^d)}.
\]
By Cauchy-Schwarz and Young's inequality, we see that for $u\in \Sigma$:
\begin{equation}\label{eq:Lest}
|L(u)| = |  \langle u, L_z u \rangle_{L^2} | \le \| \nabla u \|_{L^2(\R^d)} \, \| x u \|_{L^2(\R^d)} \le \| u\|^2_{\Sigma}<\infty.
\end{equation}

\begin{remark}
It is not clear if $\Sigma$ is the largest possible space in which one can make sense of the vorticity $L(u)$. Since the action of $L(u)$ shares a close similarity to the 
one of a magnetic field, it is conceivable that magnetic Sobolev-type spaces can also be used (see also Section \ref{sec:inverse}).
\end{remark}

The first main result of this paper is then as follows:

\begin{theorem}[Global well-posedness]\label{thm:gwp}
Let $d=2,3$.
\begin{enumerate}
\item For any $u_0\in H^1(\R^d)$, there exists a unique global in-time solution 
$u\in C(\R; H^1(\R^d))$
to \eqref{eq:nls}, depending continuously on the initial data $u_0$. 
Furthermore, $u$ obeys the conservation of mass $M(u)$, energy $E(u)$ and momentum $P(u)$, defined in \eqref{eq:energy} and \eqref{eq:mass}, respectively.

\item If in addition $u_0\in \Sigma$, then $u\in C(\R; \Sigma)$ and, in addition, $u$ conserves the angular momentum $L(u)$.
\end{enumerate}
\end{theorem}

\begin{proof} To prove item (1) of the  theorem, we note that the Peter-Paul inequality gives for any $\varepsilon > 0$
\[
|u|^4 \le \varepsilon |u|^5 + C(\varepsilon)|u|^2.
\] 
Multiplying the inequality by $-1/2$ and choosing $\varepsilon = 4/5$ yields 
 \begin{equation}\label{eq:youngs}
 -\frac{1}{2}|u|^4 + \frac{2\kappa}{5}|u|^5 \ge -\frac{C(\varepsilon)}{2}|u|^2.
 \end{equation}
 Combining this together with the conservation of energy, we have
 \begin{align*}
    E(u_0):&=  \frac{1}{2}\|\nabla u(t, \cdot)\|_{L^2(\R^d)}^2 -\frac{1}{2}\|u\|_{L^4(\R^d)}^4 + \frac{2}{5}
    \|u\|_{L^5(\R^d)}^5\\
    &\ge \frac{1}{2}\|\nabla u(t, \cdot)\|_{L^2(\R^d)}^2 -\frac{C(\varepsilon)}{2}M(u_0).
\end{align*}
This consequently yields a uniform in-time 
bound on $\|u(t, \cdot)\|_{H^1(\R^d)}$ and thus, the blow-up alternative implies that $T=\infty$.  

The proof of the conservation laws for mass, energy and momentum follows along the same lines as in \cite[Theorem~III]{Kato1987} 
(see also \cite{Oz06} for an alternative approach which does not require any additional smoothness of the solution $u$).

To prove item (2),
we consider the first derivative of the variance, which yields 
\[
\frac{d}{dt} \|x u(t, \cdot)\|_{L^2(\R^d)}^2 = 2 \IM \int_{\R^d} x \bar{u} (t, \cdot) \nabla u(t, \cdot) \;dx \le \|x u(t, \cdot)\|_{L^2(\R^d)}^2 + \|\nabla u (t, \cdot)\|_{L^2(\R^d)}^2.
\]
Here, we have used the Cauchy-Schwarz inequality together with Young's inequality. Thus, the uniform in-time boundedness of $\|\nabla u(t)\|_{L^2(\R^d)}$ from item (1) yields 
\[
\frac{d}{dt} \|x u(t, \cdot)\|_{L^2(\R^d)}^2 \le C(\varepsilon, M, E) + \|x u(t, \cdot)\|_{L^2(\R^d)}^2.
\]
Now invoking the general form of Gronwall's inequality implies that 
\[
\|x u(t, \cdot)\|_{L^2(\R^d)} \lesssim e^t.
\]
This shows that for $u_0\in \Sigma$, the solution $u(t)\in \Sigma$, for all $t\ge 0$, and the conservation law for $L(u)$ follows by the fact that 
\[
\frac{d}{dt}L(u)(t)= \frac{d}{dt}\langle u(t, \cdot), L_z u(t, \cdot) \rangle_{L^2} = 0,
\]
which can be verified by a direct computation, see, e.g.\cite{AMS}.
\end{proof}

It is clear from the proof that global well-posedness requires the inclusion of the defocusing LHY-correction. Without it, solutions to the focusing 
cubic NLS in $d=2,3$ will in general exhibit finite-time blow-up, cf. \cite{CazCourant}.


\section{Central vortex states}\label{sec:m-vortexstates} 

The partially attractive nature of the cubic-quartic nonlinearity allows for the possibility of self-bound steady 
states, i.e., solitary wave solutions to \eqref{eq:nls} of the form 
\[
u(t,x) = e^{i\omega t}\phi(x).
\]
A possible approach to imprint vorticity, is to assume that the stationary profile $\phi$ is given by a multiple of an
eigenfunction to the angular momentum operator $L_z=i \partial_\theta$, see, e.g., \cite{BBBS, BV2003, CuMa2016}. 
This corresponds to making an ansatz of the form 
 \begin{equation}
 \label{eq:vortexstates}
u(t, x) = e^{i\omega t}\phi_m(x)=\begin{cases}
\,\,\,e^{i(\omega t + m\theta)} \psi(r);\quad &d=2, \\
e^{i(\omega t + m\theta)} \psi(r,z);\quad &d=3,
\end{cases}
\end{equation}
where $m\in \Z$, $\omega\in\R$ and $r=\sqrt{x_1^2 +x_2^2}$. Note that since $L_z\phi_m = m \phi_m$, we directly obtain for the 
solution to \eqref{eq:nls}:
\[
L(u(t))=\langle u(t, \cdot), L_z u(t, \cdot)\rangle_{L^2} = m\in \Z, \ \text{for all $t\in \R$,}
\]
provided $\psi \in L^2(\R^d)$.
The class of functions $\{ \phi_m \,; \, m\in \Z\}\subset L^2(\R^d)$ is usually referred to as {\it central vortex states}. Their (in-)stability properties in 
the case with LHY correction have very recently been studied numerically, see \cite{LXCYML}.

\begin{remark}
The ansatz \eqref{eq:vortexstates} is not the only possibility to ensure non-zero vorticity of steady states. In the presence of 
an additional confining potential $V=V(x)$ of harmonic oscillator type, one can instead try to obtain ground state solutions $\phi\in \Sigma$ 
which minimize the associated Gross-Pitaevskii functional
\[
E_{\Omega}(\phi) = E(\phi) + \| V \phi \|_{L^2(\R^d)}^2 + \Omega L(\phi), \quad \Omega \in \R,
\]
subject to a mass constraint (cf. the approach in Section \ref{sec:minimizer2d} below). 
For rotation speeds $\Omega$ within a certain sub-critical range this approach has been successfully used in, e.g., \cite{ANS, RS1}.
It has the drawback, however, that it requires an additional confinement $V$ in order to ensure coercivity of the energy functional $E_{\Omega}$ on the $\Sigma$-space, and thus 
the associated ground states $\phi$ can no longer be considered as ``self-bound". Moreover, it is not clear which value of $L(\phi)\in \R$  
will be achieved by the minimizer for any given choice of $\Omega$. In particular, non-existence of vortex solutions in this context was shown in \cite{GLY}.

\end{remark}

With the ansatz \eqref{eq:vortexstates}, the vortex profile $\phi_m(x)$ solves 
\begin{equation}
 \label{eq:soliton-m}
 - \frac{1}{2}\Delta\phi_m      
 - |\phi_m|^2\phi_m + |\phi_m|^3\phi_m + \omega\phi_m=0,\quad x\in\R^d,
\end{equation}
where $d=2,3$. The associated cylindrically symmetric amplitude $\psi=\psi(r,z)$ satisfies the following equation
\begin{equation}
 \label{eq:vortices}
 - \frac{1}{2}\Delta\psi      
 - |\psi|^2\psi + |\psi|^3\psi +\left(\omega+\frac{m^2}{2r^2}\right)\psi=0.
\end{equation}
In 2D, we naturally think of $\psi=\psi(r)$ as a radially symmetric function. 
Observe that if $m=0$, equation \eqref{eq:vortices} reduces to the zero vorticity case, i.e., equation \eqref{eq:soliton-m} for $\phi_0$. Also note that 
if $\psi$ solves \eqref{eq:vortices}, then so does $e^{i \theta} \psi(\cdot + z_0)$, for $\theta \in \R$, $z_0\in \R$, due to 
gauge invariance and translation invariance w.r.t. the $z$-axis.

The energy \eqref{eq:energy} associated to solutions of the form \eqref{eq:vortexstates} becomes
\begin{equation}
\label{eq:VortexEnergy}
 E_m(\psi) = \frac{1}{2}\int_{\R^d}|\nabla \psi|^2\,dx + \frac{m^2}{2}\int_{\R^d}\frac{|\psi|^2}{r^2}\,dx-\frac{1}{2}\int_{\R^d}|\psi|^4\,dx + \frac{2}{5} \int_{\R^d}|\psi|^5\,dx.
\end{equation}
The first two terms can be re-interpreted as the kinetic and potential energies of a quantum particle under the action of a repulsive inverse-square potential $V(x) = \frac{m^2}{2r^2}$, a 
fact which we shall explore in more detail in Section \ref{sec:inverse} below.

In the following, we denote by $H^1_m(\R^d)$ the space of cylindrically symmetric functions with 
finite {\it vortex-energy} $E_m$, i.e.,
\begin{equation}
\label{eq:spaceH1m}
H^1_m(\R^d):=\left\{ u\in H^1(\R^d)\ : \ r^{-1}u \in L^2(\R^d)\ \text{and}\ u(x)=u(r,z)\right\},
\end{equation}
with the understanding that in $d=2$, we have $r=|x|$ and $u= u(|x|)$ radially symmetric. 
In particular, 
\[
H^1_m(\R^2) \equiv H^1_{m, {\rm rad}}(\R^2).
\]
The norm associated to the space defined in \eqref{eq:spaceH1m} is
\begin{equation}
\label{eq:normH1m}
\|u\|^2_{H^1_m(\R^d)}:=\|\nabla u\|^2_{L^2(\R^d)} + m^2 \|r^{-1}u\|^2_{L^2(\R^d)} + \|u\|^2_{L^2(\R^d)}.
\end{equation}
To make the notation simpler, we also introduce
\begin{equation}
\label{eq:spaceHdot1m}
\dot{H}^1_m(\R^d):=\left\{ \nabla u\in L^2(\R^d)\ : \ r^{-1}u\in L^2(\R^d)\right\},
\end{equation}
equipped with the norm
\begin{equation}
\label{eq:normdotH1m}
\|u\|^2_{\dot{H}^1_m(\R^d)}:=\|\nabla u\|^2_{L^2(\R^d)} + m^2 \|r^{-1}u\|^2_{L^2(\R^d)} .
\end{equation}

The following lemma collects some necessary conditions for the existence of central vortex states. 

\begin{lemma}[Preliminary estimates]
\label{lem:pre-est} Suppose $d=2,3$. 
\begin{enumerate}
\item Any solution $\psi \in H^1_m(\R^d)$ to \eqref{eq:vortices} satisfies the Pohozaev identities 
\begin{equation}
  \label{eq:psi1}
  \begin{aligned}
  \frac{1}{2}\|\psi\|^2_{\dot{H}^1_m(\R^d)} - \|\psi\|^4_{L^4(\R^d)}   + \|\psi\|^5_{L^5(\R^d)} +\omega \|\psi\|^2_{L^2(\R^d)}= 0,
  \end{aligned}
\end{equation}
\begin{equation}
	\label{eq:psi2a}
	\begin{aligned}
	\frac{d-2}{4d}\|\psi\|^2_{\dot{H}^1_m(\R^d)} - \frac{1}{4}\|\psi\|^4_{L^4(\R^d)}  + \frac{1}{5}\|\psi\|^5_{L^5(\R^d)} +\frac{\omega }{2} \|\psi\|^2_{L^2(\R^d)}= 0.
  \end{aligned}
\end{equation}
\item A necessary condition on the frequency $\omega \in \R$ to have a nontrivial solution $\psi\not \equiv 0$ is
\begin{equation*}
  0<\omega<\frac{25}{216}.
\end{equation*}
\item\label{lem:pre-est3} For a nontrivial solution $\psi$ to equation \eqref{eq:vortices} in 2D we have
\[
\|\psi\|_{L^2(\R^2)}^2 > \|Q_m\|_{L^2(\R^2)}^2,
\]
where $Q_m=Q_m(r)>0$ is the unique ground state solution to 
\begin{equation}\label{eq:soliton-cubic}
- \frac{1}{2}\Delta Q_m +\left(1+\frac{m^2}{2|x|^2}\right)Q_m      
 - Q_m^3 =0.
\end{equation}
 \end{enumerate}
\end{lemma}

\begin{proof}
For (1), we first assume that $\psi$ is sufficiently smooth and rapidly decaying as $r\to \infty$. Then we directly obtain \eqref{eq:psi1} by multiplying \eqref{eq:vortices} with $\bar \psi$ and integrating w.r.t. $x\in \R^d$. 
To obtain \eqref{eq:psi2a}, we instead multiply by \eqref{eq:vortices} with $x\cdot \nabla \bar \psi$ and integrate by parts. Now by taking $(d/5)\times$\eqref{eq:psi1}$-$\eqref{eq:psi2a}, we get
 \begin{equation}
 \label{eq:psi2}
 \frac{10-3d}{2}\|\psi\|^2_{\dot{H}^1_m(\R^d)} + \frac{d}{2}\int_{\R^d}|\psi|^4dx = 3\omega d\int_{\R^d}|\psi|^2dx,
 \end{equation}
for sufficiently ``nice" $\psi$, 
and a limiting argument allows us to extend this result to general $\psi$ (for details see \cite{BL83a}). 

(2) We observe that \eqref{eq:psi2} implies that $\omega >0$ (unless $\psi\equiv 0$) is necessary for nontrivial $\psi$. Now we find the upper bound for $\omega$, for which we introduce the function
\[
F(s) = \frac{1}{4}s^4 - \frac{1}{5}s^5,
\] 
and set 
\[
\omega^* = \sup\Big\{\omega > 0;\quad \frac{\omega}{2}s^2 - F(s) < 0\,\,\text{for some}\,\,s>0 \Big\}.
\]
A (not so direct) computation then shows that $\omega^*=\frac{25}{216}$. In particular, if $\omega > \frac{25}{216}$, we have
\[
-\frac{1}{4}|\psi(x)|^4 + \frac{1}{5}|\psi(x)|^5 + \frac{\omega}{2}|\psi(x)|^2\ge 0,\quad \forall x\in\R^d. 
\]
\eqref{eq:psi2a} then implies that $\psi\equiv 0$. 

(3) We fix $d=2$ and introduce
\[
\gamma = \frac{\|\psi\|^4_{L^4(\R^2)}}{\|\psi\|^2_{\dot{H}^1_m(\R^2)}}.
\]
We then rewrite \eqref{eq:psi1} as
\[
\left(\frac{1}{2}-\gamma\right)\|\psi\|^2_{\dot{H}^1_m(\R^2)} + \|\psi\|^5_{L^5(\R^2)} + \omega \|\psi\|^2_{L^2(\R^2)} = 0.
\]
Similarly, we also rewrite \eqref{eq:psi2a} using $\gamma$
\[
-\frac{\gamma}{2}\|\psi\|^2_{\dot{H}^1_m(\R^2)} + \frac{2}{5}\|\psi\|^5_{L^5(\R^2)} + \omega\|\psi\|^2_{L^2(\R^2)} =0.
\]
Combining the above expressions, we get 
\[
\|\psi\|^5_{L^5(\R^2)} = \frac{5(\gamma-1)}{6}\|\psi\|^2_{\dot{H}^1_m(\R^2)}.
\]
which implies that for a solution $\psi\not =0$ to equation \eqref{eq:vortices}, we have $\gamma >1$, i.e., $\|\psi\|^4_{L^4(\R^2)} > \|\psi\|^2_{\dot{H}^1_m(\R^2)}$. We now invoke the 
following Gagliardo-Nirenberg-type inequality, see \cite{Zheng}:
\begin{equation}\label{eq:GNip}
\|f\|^4_{L^4(\R^2)}\leq C_m\|f\|_{\dot{H}^1_m(\R^2)}^2\|f\|_{L^2(\R^2)}^2, 
\end{equation}
It is known that the optimal constant $C_m>0$ is given by $C_m= \|Q_m\|_{L^2(\R^2)}^{-2}$, where $Q_m$ is the radial cubic ground state solution satisfying \eqref{eq:soliton-cubic}. 
In view of \eqref{eq:GNip}, we obtain that $\psi$ satisfies $\|\psi\|_{L^2(\R^2)}^2 > \|Q_m\|_{L^2(\R^2)}^2,$ as desired. 
\end{proof}


\section{Existence of vortex states in $2$D}\label{sec:minimizer2d}

In this section, we shall rigorously prove the existence of a class of central vortex states, i.e., solutions $\psi\in H^1_m(\R^2)$ to \eqref{eq:vortices}. To this end, we  
search for minimizers of the associated vortex-energy $E_m(\psi)$ for a given (fixed) mass $\rho>0$. 
More precisely, we consider the following minimization problem
  \begin{equation}\label{eq:Emin2d} 
 e(\rho)=    \inf \{ E_m(\varphi)\ : \  \varphi\in H^1_m(\R^2),  \ M(\varphi) = \rho \}.
  \end{equation}
Assuming that $e(\rho)>-\infty$, we denote the associated set of (constrained) energy minimizers by
  \begin{equation}\label{eq:Eminset2d}
    \mathcal  E(\rho):=\{\psi \in H^1_m(\R^2),\  M(\psi)=\rho,\ E_m(\psi) = e(\rho)\},
  \end{equation}
 and call $\psi\in \mathcal E(\rho)$ a {\it vortex ground state}. 
Note that for any $\psi\in \mathcal E(\rho)$ there exists a Lagrange multiplier $\omega$, such that
\[
dE_m(\psi) + \omega  \, dM(\psi) =0.
\]    
Thus, $\psi\in \mathcal E(\rho)$  solves \eqref{eq:vortices} for some $\omega \in (0, \tfrac{25}{216})$. The following theorem will put this idea on firm grounds.

\begin{theorem}[Existence of vortex-ground states in 2D]
\label{thm:exist2d}
Let $\rho_* = \| Q_m\|_{L^2}^2$, where $Q_m$ is the ground state solution to \eqref{eq:soliton-cubic}. Then for all $\rho>\rho_*$ the set of energy minimizers  $\mathcal E(\rho)\not = \emptyset$. Moreover, for any $\rho>\rho_*$ there exists at least one 
minimizer $\psi\in \mathcal E(\rho)$ which is positive, radial and non-increasing in $|x|=r$.
\end{theorem}

\begin{proof}
The proof that there exists a minimizer for $\rho > \rho_*$ follows by establishing the compactness of any minimizing sequence for \eqref{eq:Emin2d}. 
The fact that we are in the radial setting, which allows for compact embeddings, helps us to simplify the approach when compared to the 3D case: 

\smallskip
{\bf Step 1:} We start by showing that for all $\rho > \rho_*$: $-\infty<e(\rho)<0$, where $e(\rho)$ is defined in \eqref{eq:Emin2d} (cf. \cite[Theorem 6]{LewinRotaNodari2020} 
for the proof in the case $m = 0$). We invoke \eqref{eq:youngs} to write
\begin{equation}\label{eq:lowbdd}
E_m(\varphi)\ge \frac{1}{2}\|\varphi\|_{\dot{H}^1_m(\R^2)}^2 -C(\varepsilon,\rho),
\end{equation}
where $C(\varepsilon,\rho) = \frac{C(\varepsilon)}{2}\,\rho > 0$ and $\varepsilon=4/5$. Therefore, $e(\rho) > - C(\varepsilon,\rho) > -\infty$. Now for $\lambda>0$, we 
let $\varphi_{\lambda} = \lambda\varphi(\lambda x)$, the $L^2$-invariant scaling, and write 
\[
E_m(\varphi_{\lambda}) = \frac{\lambda^2}{2}\left(\|\varphi\|_{\dot{H}^1_m(\R^2)}^2 - \|\varphi\|^4_{L^4(\R^2)} + \frac{4\lambda}{5}\|\varphi\|^5_{L^5(\R^2)}\right).
\]
Invoking the sharp Gagliardo-Nirenberg inequality \eqref{eq:GNip} 
\[
\|\varphi\|_{L^4(\R^2)}^4 \leq \left(\frac{\|\varphi\|_{L^2(\R^2)}}{\|Q_m\|_{L^2(\R^2)}}\right)^2\|\varphi\|_{\dot{H}^1_m(\R^2)}^2,
\]
with the observation that $\|\varphi\|_{L^2(\R^2)}^2 > \|Q_m\|_{L^2(\R^2)}^2$ from Lemma \ref{lem:pre-est}(\ref{lem:pre-est3}), we may choose a profile 
$\varphi\in H^1_m(\R^2)$ such that the $\lambda$-independent terms inside the parenthesis become negative, e.g., take
\[
\varphi = \sqrt{\frac{\rho}{M\(Q_m\)}}\,Q_m, \ \text{for a sufficiently small}\  \lambda > 0.
\]
 Thus, one can find a $\rho_*$ such that $E_m(\varphi_\lambda)$ is negative for all $\rho > \rho_*$. We now use the value of the best constant $C_m$ (as that value will be unique) 
 in the Gagliardo-Nirenberg inequality \eqref{eq:GNip} to identify $\rho_*$ via the relation $C_m = \big(\rho_*\big)^{-1}$. 
 Thus, we denote by $\varphi_*$ any vortex state solution of \eqref{eq:vortices} (we do not know whether $\varphi_*$ is unique), and express $\rho_*$ via
 \[
M(\varphi_*) = \rho_* = \big(C_m\big)^{-1} = \|Q_m\|_{L^2(\R^2)}^2,
\] 
which is uniquely obtained from the best constant $C_m$ in \eqref{eq:GNip}. Observe that the function $Q_m$ is not a minimizer for \eqref{eq:Emin2d} because $Q_m$ does not solve the appropriate Euler-Lagrange equation \eqref{eq:vortices}.

\smallskip
{\bf Step 2:} In what follows, we consider a fixed mass $\rho>\rho_* = \|Q_m\|_{L^2(\R^2)}^2$, such that $-\infty < e(\rho) < 0$ and let $(\varphi_n)_{n\ge 1} \subset H^1_m(\R^2)$ be a minimizing sequence to \eqref{eq:Emin2d}, i.e.,
\[
\liminf\limits_{n \to \infty} E_m\(\varphi_n\) = e(\rho)\quad \text{and} \quad \|\varphi_n\|_{L^2}^2 = \rho, \ \forall n\geq 1.
\]
Therefore, there exists a positive constant $c$ such that
\[
E_m\(\varphi_n\) < e(\rho) + c,\ \forall n\geq 1.
\]
Inserting \eqref{eq:lowbdd} into the above estimate, we get
\[
\frac{1}{2}\|\varphi\|_{\dot{H}^1_m(\R^2)}^2 < e(\rho) + c + C(\varepsilon,\rho),
\]
i.e., $(\varphi_n)_{n\ge 1}$ is a bounded radial sequence in $H^1_m(\R^2)$. Therefore, for $\psi\in H^1_m(\R^2)$, there exists a subsequence, still denoted by $\varphi_n$, such that $\varphi_n$ converges weakly to $\psi$ in $H^1_m(\R^2)$. By the radial compact embedding $H^1_m(\R^2) \hookrightarrow L^p(\R^2)$ for $2<p<\infty$, see, e.g., \cite[Theorem 1(2)]{PL86}, $\varphi_{n}$ converges strongly to $\psi$ in $L^p(\R^2)$.

Denote 
\[
S(\varphi) = -\Delta \varphi + \frac{m^2}{|x|^2}\varphi + \varphi - |\varphi|^2\varphi + |\varphi|^3\varphi,
\]
and observe that
\begin{equation*}
\begin{aligned}
& \|\varphi_n-\psi\|_{H^1_m(\R^2)}^2 = \langle S(\varphi_n) - S(\psi), \varphi_n-\psi\rangle_{L^2} \\
&+ \int_{\R^2}\(|\varphi_n|^2\varphi_n - |\psi|^2\psi\)\(\bar{\varphi}_n-\bar{\psi}\)dx
- \int_{\R^2}\(|\varphi_n|^3\varphi_n - |\psi|^3\psi\)\(\bar{\varphi}_n-\bar{\psi}\)dx.
\end{aligned}
\end{equation*}
We deduce by H\"older's inequality that 
\[
\left|\int_{\R^2}\(|\varphi_n|^2\varphi_n - |\psi|^2\psi\)\(\bar{\varphi}_n-\bar{\psi}\)dx\right| \le \(\|\varphi_n\|_{L^4(\R^2)}^2 + \|\psi_n\|_{L^4(\R^2)}^2\)\|\varphi_n-\psi\|_{L^4(\R^2)}^2,
\]
which (by strong convergence in $L^p(\R^2)$) goes to zero, as $n\rightarrow \infty$.  Similarly,
\[
\left|\int_{\R^2}\(|\varphi_n|^2\varphi_n - |\psi|^2\psi\)\(\bar{\varphi}_n-\bar{\psi}\)dx\right| \le \(\|\varphi_n\|_{L^5(\R^2)}^3 + \|\psi_n\|_{L^5(\R^2)}^3\)\|\varphi_n-\psi\|_{L^5(\R^2)}^2
\]
tends to zero, as $n\rightarrow \infty$. We now show that for all $\eta\in C^{\infty}_0(\R^2)$
\begin{equation}\label{eq:convg0}
\langle S(\varphi_n) - S(\psi), \eta\rangle_{L^2} \rightarrow 0\ \text{as}\ n\rightarrow \infty.
\end{equation}
Indeed, we have
\begin{equation*}
\begin{aligned}
 \langle S(\varphi_n) - S(\psi), \eta\rangle_{L^2} = &\int_{\R^2}\(\nabla\varphi_n-\nabla\psi\)\nabla\bar{\eta}\,dx + \int_{\R^2}\(1+\frac{m^2}{|x|^2}\)\(\varphi_n-\psi\)\bar{\eta}\,dx\\
 &-  \int_{\R^2}\(|\varphi_n|^2\varphi_n - |\psi|^2\psi\)\bar{\eta}\,dx + \int_{\R^2}\(|\varphi_n|^3\varphi_n - |\psi|^3\psi\)\bar{\eta}\,dx.
\end{aligned}
\end{equation*}
 Since $\varphi_n$ converges weakly to $\psi$ in $H^1_m(\R^2)$ and strongly in $L^p(\R^2)$, \eqref{eq:convg0} follows.  This implies that $\varphi_n$ converges strongly to $\psi$ in $H^1_m(\R^2)$. Therefore, using the compactness together with the strong convergence, we get
 \[
E_m(\psi) = \lim_{n\to \infty}E_m(\varphi_{n})= e(\rho).
\]
This shows the existence of a minimizer for $e(\rho)$.  We point out that since we are considering the minimization problem along sequences $(\varphi_n)_{n\ge 1} \subset H^1_m(\R^2)\equiv H^1_{m, {\rm rad}} (\R^2)$, and thus over radially symmetric functions, we obtain a radially symmetric minimizer which is non-increasing w.r.t. $|x|$. Finally, since $E_m(\varphi)=E_m(|\varphi|)$ (in view of \cite[Theorem 7.8]{LiLo}) we may assume that the radially symmetric minimizer is positive.
\end{proof}

\begin{remark} The existence of constrained energy minimizers in the vortex-free case $m=0$ was recently proved in \cite{LewinRotaNodari2020} 
for a general class of competing (energy-subcritical) power-law nonlinearities
(see also \cite{CaSp1} for a detailed study of the cubic-quintic case in dimensions $d\le 3$). 
Earlier results in this direction are in the case of only a single (power-law) nonlinearity and can be found in, e.g., \cite{BL83a, BV2003, PL284a}.
In there, the authors seek to minimize the kinetic part of the energy under a constrained given by the potential energy plus the mass.
The drawback of this approach is that it does not allow to infer orbital stability even if $m=0$, in contrast to our approach (see below). 
\end{remark}


\section{Connection to the NLS with inverse-square potential} \label{sec:inverse}

An alternative way to think about the vortex state equation \eqref{eq:vortices} in $2$D is to re-interpret its solutions $\psi$ as 
steady states of the NLS with (repulsive) {\it inverse-square potential}. More precisely, consider
\begin{equation}\label{eq:iPnls}
\left\{
\begin{aligned}
 & i\d_t v +\frac{1}{2}\Delta v  - \frac{m^2}{2|x|^2} v= - |v|^2v +  |v|^3v,\quad  t \in \R,\ x\in \R^2,\\
  & v_{\mid t=0}  = v_0(x),
\end{aligned}  
\right.
\end{equation}
and observe that the associated solution $v(t,x)$ formally conserves the mass $M(v)$ and the vortex-energy $E_m(v)$. 
In addition, it is easy to see that if the initial data is radially symmetric, i.e., if 
$v_0=v_0(|x|)$, then so is the solution $v(t,x)$ for all $t\in \R$. 
Seeking nontrivial {\it radially symmetric solitary wave solutions} to \eqref{eq:iPnls} of the form 
\[
v(t,x) = e^{i\omega t}\psi (x), \quad \omega\in\R,  
\]
therefore lead to the study of vortex-profiles $\psi=\psi(|x|)$, solutions to \eqref{eq:vortices}. Moreover, the fact that 2D vortex ground states $\psi$ 
are obtained by minimizing the vortex-energy $E_m$, will allow us to infer the orbital stability of the set $\mathcal E(\rho)$ under the dynamics of \eqref{eq:iPnls}. 

To advance this line of reasoning, we start with a well-posedness result for equation \eqref{eq:iPnls} in $H^1_m(\R^2)$. To this end, we rewrite 
equation \eqref{eq:iPnls} using Duhamel's formula, i.e.,
\begin{equation}\label{eq:duhamel-ip}
v(t) = e^{i\frac{t}{2}\mathcal{L}_m}v_0 - i \int_0^t e^{i\frac{t-s}{2}\mathcal{L}_m}f(v(s)) \,ds =: \Phi_m(v(t)).
\end{equation}
Here, and in the following, we denote, as before, $f(z)=-|z|^2z+|z|^3z$, and
\[
\mathcal{L}_m := -\Delta + \frac{m^2}{|x|^2}, \quad m\in \Z.
\] 

Recall that the space 
$H^1_m(\R^2)\equiv H^1_{m, {\rm rad}} (\R^2)$ consists of radially symmetric functions. This allows us to invoke the arguments of \cite{GX2020}, 
where it is shown that for radial functions, the Sobolev space $\dot{H}^1_m(\R^2)$ is equivalent to the magnetic Sobolev space $\dot{H}^1_A(\R^2)$. 
The latter is defined as the completion of $C_0^\infty(\R^2\setminus \{ 0 \})$ under the semi-norm
\[
\| f \|^2_{\dot{H}^1_A(\R^2)} : = \int_{\R^2} | \nabla_A f |^2 \, dx,
\]
where $\nabla_A = \nabla + iA(x)$ and the vector potential is $A(x)=m \big(-\frac{x_2}{|x|^2},\frac{x_1}{|x|^2}\big)$. It is then proved in \cite[Section 2.1]{GX2020} that 
\[
\| f \|^2_{H^1_A(\R^2)} \simeq \| f \|^2_{H^1_m(\R^2)}, \quad \text{for radial $f$.}
\]
In addition, it is also shown that the linear magnetic Schr\"odinger operator 
\[
\(-i\nabla +A(x)\)^2 f = \mathcal L_m f, \quad \text{for radial $f$.}
\]
In turn this allows one to infer Strichartz estimates associated to the linear Schr\"odinger propagator with 
inverse square-potential $S_m(t) = e^{i\frac{t}{2}\mathcal{L}_m}$, see \cite{GX2020} and the references given 
therein. Using these, we are able to prove the following well-posedness result, which generalizes the one in \cite{GX2020}:

\begin{proposition}[Well-posedness with inverse-square potential]
For any $v_0 \in H^1_m(\R^2)$, there exists a unique global in-time solution $v \in C(\R;H^1_m(\R^2))$ to \eqref{eq:iPnls}, depending continuously on the initial data $v_0$. 
In addition, $v(t, \cdot)$ is radially symmetric and conserves the mass $M(v)$ and the vortex-energy $E_m(v)$.
\end{proposition}

\begin{proof} Denote $\nabla_m = \nabla + m|x|^{-1}$. We shall first show that for any $v_0\in H^1_m(\R^2)$, there exists a $T>0$ and a 
unique solution $v(t)$ to the integral equation \eqref{eq:duhamel-ip}, depending continuously on $v_0$, such that
\[
v,\nabla_m v\in \mathcal{X}_m := C([0, T];L^2(\R^2))\cap L^{4}((0, T); L^{1,4}(\R^2))\cap L^{\frac{10}{3}}((0, T); L^{5}(\R^2)) .
\]
This follows by arguments similar to those given in the proof of Proposition \ref{prop:LWP}
To this end, let $M>0$, to be specified later, and 
\begin{equation*}
	\widetilde{B}_M=\{v\in \mathcal{X}_m\,:\|v\|_{L^{\infty}(I;H^1_m)}\leq M \}.
\end{equation*}
We shall prove that the operator $\Phi_m(v(t))$ defined in \eqref{eq:duhamel-ip} is a contraction on the set $\widetilde{B}_M$. Using Strichartz estimates (cf. \cite[Proposition 2.2]{GX2020}), we obtain
\begin{equation}
\begin{aligned}\label{Lwp1ip}
	\|\Phi_m(v(t))\|_{L^{\infty}(I;L^{2})} &+ \|\Phi_m(v(t))\|_{L^{4}(I;L^{4})} + \|\Phi_m(v(t))\|_{L^{\frac{10}{3}}(I;L^{5})} \\
	&\lesssim \| v_0\|_{L^2(\R^d)} + \||v|^2v\|_{L^{\frac{4}{3}}(I;L^{\frac{4}{3}})} + 
	 \||v|^3v\|_{L^{\frac{10}{7}}(I;L^{\frac{5}{4}})},
\end{aligned}
\end{equation}
and
\begin{equation}
\begin{aligned}\label{Lwp2ip}
	 \|\nabla_m\Phi_m(v(t))&\|_{L^{\infty}(I;L^{2})}  + \|\nabla_m\Phi_m(v(t))\|_{L^{4}(I;L^{4})} + \|\nabla_m\Phi_m(v(t))\|_{L^{\frac{10}{3}}(I;L^{5})} \\
	  \lesssim &\,\,\| \nabla_m v_0\|_{L^2(\R^d)} + \||v|^2 \nabla v\|_{L^{\frac{4}{3}}(I;L^{\frac{4}{3}})} + 
	\||v|^3 \nabla v\|_{L^{\frac{10}{7}}(I;L^{\frac{5}{4}})}\\
	& \  +  \|m|x|^{-1}|v|^2  v\|_{L^{\frac{4}{3}}(I;L^{\frac{4}{3}})} + 
	 \|m|x|^{-1}|v|^3 v\|_{L^{\frac{10}{7}}(I;L^{\frac{5}{4}})}.
\end{aligned}
\end{equation}
 Using H\"older's and Sobolev's inequality, we get
\begin{align*}
	\||v|^2v&\|_{L^{\frac{4}{3}}(I;L^{\frac{4}{3}})} + 
	 \||v|^3v\|_{L^{\frac{10}{7}}(I;L^{\frac{5}{4}})} \\
	&\lesssim T^{\frac{3}{4}}\|v\|^2_{L^{\infty}(I;L^{8})} \|v\|_{L^{\infty}(I;L^{2})} 
	+ T^{\frac{7}{10}}\| v\|^3_{L^{\infty}(I;L^{10})} \|v\|_{L^{\infty}(I;L^{2})}\\
	&\lesssim T^{\frac{3}{4}}\|\nabla v\|^2_{L^{\infty}(I;L^{2})} \|v\|_{L^{\infty}(I;L^{2})} 
	+ T^{\frac{7}{10}}\| \nabla v\|^3_{L^{\infty}(I;L^{2})} \|v\|_{L^{\infty}(I;L^{2})},
	\end{align*}
	 as well as
\begin{align*}
	\||v|^2 \nabla v\|_{L^{\frac{4}{3}}(I;L^{\frac{4}{3}})} + 
	\||v|^3 \nabla v\|_{L^{\frac{10}{7}}(I;L^{\frac{5}{4}})}
	\lesssim T^{\frac{3}{4}}\|\nabla v\|^3_{L^{\infty}(I;L^{2})}
	+ T^{\frac{7}{10}}\|\nabla v\|^4_{L^{\infty}(I;L^{2})},
\end{align*}
and
\begin{align*}
  \||x|^{-1}&|v|^2  v\|_{L^{\frac{4}{3}}(I;L^{\frac{4}{3}})} + 
 \||x|^{-1}|v|^3 v\|_{L^{\frac{10}{7}}(I;L^{\frac{5}{4}})}\\
	&\lesssim T^{\frac{3}{4}}\|\nabla v\|^2_{L^{\infty}(I;L^{2})} \||x|^{-1}v\|_{L^{\infty}(I;L^{2})} 
	+ T^{\frac{7}{10}}\| \nabla v\|^3_{L^{\infty}(I;L^{2})} \||x|^{-1}v\|_{L^{\infty}(I;L^{2})}.
	\end{align*}
Combining \eqref{Lwp1ip} and \eqref{Lwp2ip} together with the above three estimates, we obtain 
\begin{align}\label{Lwp5ip}\notag
	 \|\Phi_m(v(t))\|_{L^{\infty}(I;H^1_m)} +\:& \: \|(1+\nabla_m)\Phi_m(v(t))\|_{L^{4}(I;L^{4})} + \|(1+\nabla_m)\Phi_m(v(t))\|_{L^{\frac{10}{3}}(I;L^{5})}\\\notag
	 &\lesssim \|v_0\|_{H_m^1(\R^2)}+T^{\frac{3}{4}}\|\nabla v\|^2_{L^{\infty}(I;L^{2})} \|v\|_{L^{\infty}(I;H^{1}_m)} \\\notag
	 &\qquad\qquad\qquad\qquad+T^{\frac{7}{10}}\|\nabla v\|^3_{L^{\infty}(I;L^{2})} \|v\|_{L^{\infty}(I;H^{1}_m)}\\
	 &\le C\|v_0\|_{H^1_m(\R^2)} + CT^{\frac{3}{4}}M^3 + CT^{\frac{7}{10}}M^4,
\end{align}
for $v\in \widetilde{B}_M$. Set $M=2C\|v_0\|_{H^1_m(\R^2)}$ and take $T$ such that
\begin{align}\label{Lwp6ip}
	CT^{\frac{3}{4}}M^{2} + CT^{\frac{7}{10}}M^{3} \le \frac{1}{2},
\end{align}
yielding that \eqref{Lwp5ip} is bounded by $M$. Therefore, 
$\Phi_m$ maps $\widetilde{B}_M$ to itself for some time $T=T(\|v_0\|_{H^1_m})>0$. Similarly, one finds that for $v,w\in \widetilde{B}_M$:
\[
d(\Phi_m(v(t)),\Phi_m(w(t))) \lesssim \big(T^{\frac{3}{4}}M^{2} + T^{\frac{7}{10}}M^{3}\big)d(v,w),
\]
which together with the smallness condition \eqref{Lwp6ip} implies that $\Phi_m$ is a contraction on $\widetilde{B}_M$. In turn, this yields the existence of 
a unique solution $v \in C([0, T];H^1_m(\R^2))$, and continuous dependence on the initial data then follows by classical arguments. 

Finally, the proof that $v(t)$ can be extended for all $T>0$, yielding a global in-time solution, follows by invoking the conservation laws of mass $M(v)$ and vortex-energy $E_m(v)$, together with 
the same reasoning as in the proof of Theorem \ref{thm:gwp}(1). 
\end{proof}

Next, we recall that the vortex ground-states in Theorem \ref{thm:exist2d} are obtained by minimizing the energy $E_m$, subject to a mass constrained. The fact that $E_m$ is also 
conserved under the flow of \eqref{eq:iPnls} allows us to infer their orbital stability under this flow.

\begin{proposition}[Orbital stability of vortex-ground states]\label{prop:stability}
The set $\mathcal E(\rho)$, defined in \eqref{eq:Eminset2d}, is orbitally stable in $H^1_m(\R^2)$ under the dynamics with inverse square-potential. More precisely, for all $\eps>0$, there exists $\delta>0$ such that if
 $v_0\in H_{m}^1(\R^2)$ satisfies
 \[\inf_{\psi\in \mathcal E(\rho)}\|v_0-\psi\|_{H^1_{m}(\R^2)}\le \delta,\]
  then the solution to \eqref{eq:iPnls} with $v_{\mid t=0}=v_0$ satisfies
  \begin{equation*}
    \sup_{t\in \R}\inf_{\psi\in \mathcal E(\rho)}\left\|v(t)-\psi\right\|_{H^1_{m}(\R^2)}\le \eps. 
  \end{equation*}
\end{proposition}

\begin{proof}
We follow the approach of \cite{CaLi}: Assume, by contradiction, that
there exist a sequence of (radially symmetric) initial data $(v_{0,n})_{n\in \N}\subset H^1_{m}(\R^2)$ with
\begin{equation}\label{eq:stability1}
  \lim_{n\to \infty}\|v_{0,n}-\psi\|_{H^1_{m}(\R^2)}= 0,
\end{equation}
and a sequence $(t_n)_{n\in \N}\subset \R$, such that the sequence of solutions $v_n(t_n)$ to \eqref{eq:iPnls} associated to $v_{0,n}$ satisfies
\begin{equation}\label{eq:stability2}
  \inf_{\psi\in \mathcal E(\rho)}\left\|v_n(t_n) -
    \psi\right\|_{H^1_{m}(\R^2)}>\eps,
\end{equation}
for some $\eps>0$.
In view of \eqref{eq:stability1}, we find that, on the one hand:
\begin{equation*}
  \lim_{n\to \infty}\|v_{0,n}\|_{L^2(\R^2)}^2 = \|\psi\|_{L^2(\R^2)}^2\equiv \rho,\quad
  \lim_{n\to \infty}E_m(v_{0,n}) = E_m(\psi)=e(\rho).
  \end{equation*}
On the other hand, the conservation laws for mass and energy imply
\begin{equation*}
 \begin{aligned}
 \lim_{n\to \infty}\|v_n(t_n)\|_{L^2(\R^2)}^2 &= \lim_{n\to \infty}\|v_{0,n}\|_{L^2(\R^2)}^2\equiv \rho,\\
   \lim_{n\to \infty}E_m(v_n(t_n))= &\lim_{n\to \infty}E_m(v_{0,n})=e(\rho).
\end{aligned}
\end{equation*}
Thus, up to rescaling the mass by an appropriate factor, $(v_n(t_n))_n$ is a minimizing sequence for the constrained minimization problem
\eqref{eq:Emin2d}. By the arguments outlined in the proof of Proposition \ref{thm:exist2d}, there exist a
strongly convergent subsequence $(v_n(t_n))_n\subset H^1_m(\R^2)$ with limit $\psi$. This is a contradiction to \eqref{eq:stability2}. 
\end{proof}

Central vortex states with $m\not =0$ can be seen as excited states for equation \eqref{eq:soliton-m}. In the case of a single (focusing) power-law nonlinearity, they 
are known to be orbitally unstable, in general, even though some of them appear to be 
spectrally stable. For a broader discussion on all this we see, e.g., \cite{CGNT, Mi}, and in particular the introduction of \cite{CuMa2016}. In the case with 
competing power-law nonlinearities, numerical simulations seem to indicate the existence of a critical value $\omega_{\rm cr}$, above 
which vortex solution are spectrally stable \cite{PeWa}, but a complete picture is still missing. 
In view of our orbital stability result, one might conjecture that the orbital instability of 2D vortiex states is essentially due to the difference between Schr\"odinger flows 
with and without repulsive inverse square potential.

An immediate consequence of Proposition \ref{prop:stability} 
is the orbital stability of constrained energy minimizers in the case of zero vorticity, i.e., $m=0$, where the dynamics of \eqref{eq:iPnls} reduces to the one 
of the original NLS \eqref{eq:nls}. The corresponding minimization problem reduces to
\[
e_0(\rho) = \inf\{E_0(\varphi)\ : \ \varphi\in H_{\rm rad}^1(\R^2),\ M(\varphi) = \rho\}.
\]
The associated set of energy minimizers, usually called {\it energy ground states}, is denoted by $\mathcal E_0(\rho)$. The members of this set are known to 
be real-valued and positive, cf. \cite{LewinRotaNodari2020}. In particular,  
they do not carry any vorticity. 

\begin{corollary}
If $m=0$, the set $\mathcal E_0(\rho)$ is orbitally stable in $H_{\rm rad}^1(\R^2)$ under the flow of \eqref{eq:nls}.
\end{corollary}

\begin{remark}
Note that if $\psi\in \mathcal E_0(\rho)$, then so is $e^{i \theta} \psi( \cdot + x_0)$, for $x_0\in \R^2$, $\theta \in \R$, due to gauge and translation invariance. 
In contrast to the case with only a single (focusing) 
power-law nonlinearity, it is not known whether $\mathcal E_0(\rho)$ consists of a unique element, up to the aforementioned symmetries. Uniqueness, up to symmetries, is usually  
obtained by following the approach of minimizing, for any given frequency $\omega\in (0, \tfrac{25}{216})$, the unconstrained action functional 
$$S(\phi) = E_0(\phi) + \omega M(\phi),$$ 
over all nontrivial $\phi\in H^1(\R^2)$ solutions to \eqref{eq:soliton-m} see, e.g. \cite{CazCourant}. A drawback of these action minimizers 
is that their orbital stability cannot be established along the same lines as before. In particular, the map $\omega \mapsto \rho$ is not necessarily one-to-one.
For the case with competing power-law nonlinearities, it has recently been shown that there exist least action solutions to \eqref{eq:soliton-m}, 
which are not equal to constrained energy minimizers, see \cite{CKS, JJL1, JJL2, LewinRotaNodari2020} for a broader discussion on all this.
\end{remark}


\section{Existence of vortex states in 3D}\label{sec:minimizer3d}

We now turn to the question of existence of vortex ground-states in 3D. To this end, we consider a similar minimization problem as in Section \ref{sec:minimizer2d}:
  \begin{equation}\label{eq:Emin} 
 e(\rho)=    \inf \{ E_m(\varphi)\ : \  \varphi\in H^1_m(\R^3),  \ M(\varphi) = \rho \}.
  \end{equation}
As before, we shall characterize the corresponding set of (constrained) energy minimizers by
  \begin{equation*}
    \mathcal  E(\rho):=\{\psi \in H^1_m(\R^3),\  M(\psi)=\rho,\ E_m(\psi) = e(\rho)\}.
  \end{equation*}
 Recall $H^1_m(\R^3)$ is the space of cylindrically symmetric functions of the form $\psi(r,z)$ with $r=\sqrt{x_1^2+x_2^2}$. We then have an existence result in 3D, similar to Theorem \ref{thm:exist2d}.

\begin{proposition}[Existence of vortex-ground states in 3D]
\label{prop:exist3d}
There exists a $\rho_*>0$, such that for all $\rho>\rho_*$ the set of of energy minimizers  $\mathcal E(\rho)\not = \emptyset$. Moreover, for any $\rho>\rho_*$ there exists at least one 
minimizer $\psi\in \mathcal E(\rho)$ which is  
positive, radial in the first two coordinates, i.e., $\psi=\psi(r,z)$ and non-increasing in $|z|$.
\end{proposition}

\begin{proof} The proof follows by invoking a concentration-compactness method of \cite{PL284a}, but will require some adjustments due to working with cylindrically symmetric functions in $d=3$. 

\smallskip
{\bf Step 1:} We start by recollecting that $e(\rho) > -\infty$, which follows from \eqref{eq:youngs}. 
Now we consider a different scaling from 2D case, $\varphi_{\rho} = \varphi(\rho^{-1/3}x)=\varphi(\rho^{-1/3}r,\rho^{-1/3}z)$ such that $\|\varphi_{\rho}\|^2_{L^2(\R^3)} = \rho^2$ whenever $\|\varphi\|^2_{L^2(\R^3)} = \rho$. We then write 
\[
E_m(\varphi_{\rho}) = \frac{\rho^{\frac{1}{3}}}{2}\|\varphi\|_{\dot{H}^1_m(\R^3)}^2 - \frac{\rho}{2}\|\varphi\|^4_{L^4(\R^3)} + \frac{2\rho}{5}\|\varphi\|^5_{L^5(\R^3)},
\]
and compute
 \[
 \frac{\partial^2 E_m(\varphi_\rho)}{\partial\rho^2} = -\frac{1}{9\rho^{\frac{5}{3}}}\|\varphi\|_{\dot{H}^1_m(\R^3)}^2 < 0.
 \] 
 Thus, one can find a  $\rho_*$ such that $E_m(\varphi_\rho)$ is strictly decreasing and hence, negative for all $\rho > \rho_*$.  
We now characterize the mass constraint threshold $\rho_*$ via Gagliardo-Nirenberg type inequality (a straightforward adaption of the argument presented in \cite{LewinRotaNodari2020}). Indeed, from the concavity of $E_m$, we have that $E_m$ is non-negative for all $\varphi \in H^1_m(\R^3)$ such that $\rho < \rho_*$. This yields 
\[
\frac{1}{2}\int_{\R^3}|\varphi|^4\,dx\le \frac{1}{2}\int_{\R^3}|\nabla \varphi|^2\,dx + \frac{m^2}{2}\int_{\R^3}\frac{|\varphi|^2}{r^2}\,dx + \frac{2}{5} \int_{\R^3}|\varphi|^5\,dx.
\]
Substituting $\varphi$ with $ \sigma^{\frac{3}{2}}\varphi(\sigma x)$ in the above inequality gives 
\[
\frac{\sigma^3}{2}\int_{\R^3}|\varphi|^4\,dx \le \frac{\sigma^2}{2}\(\int_{\R^3}|\nabla \varphi|^2\,dx + m^2\int_{\R^3}\frac{|\varphi|^2}{r^2}\,dx\) + \frac{2\sigma^\frac{9}{2}}{5}\int_{\R^3}|\varphi|^5\,dx.
\]
Optimizing the last inequality with respect to $\sigma$, we get
\[
\|\varphi\|_{L^4(\R^3)}^4 \le \frac{7}{\sqrt[5]{3^35^2}} \|\varphi\|_{\dot{H}^1_m(\R^3)}^{\frac{6}{5}}\|\varphi\|_{L^5(\R^3)}^2.
\]
We now consider any vortex state solution $\varphi_*$ of the equation \eqref{eq:vortices} such that $M(\varphi_*) = \rho_*$. We normalize $\varphi$ with respect to $L^2$ norm, i.e,  replace $\varphi$ with $\displaystyle\frac{(\rho_*)^{\frac{1}{2}}}{\|\varphi\|_{L^2(\R^3)}}\varphi$ in the last estimate, which gives
\[
\|\varphi\|_{L^4(\R^3)}^4 \leq \frac{7}{\sqrt[5]{3^35^2}}\(\frac{\|\varphi\|_{L^2(\R^3)}}{(\rho_*)^{\frac{1}{2}}}\)^{\frac{4}{5}} \|\varphi\|_{\dot{H}^1_m(\R^3)}^{\frac{6}{5}}\|\varphi\|_{L^5(\R^3)}^2.
\]
Therefore, the best constant in the Gagliardo-Nirenberg inequality
\begin{equation}\label{eq:GN3d}
\|\varphi\|_{L^4(\R^3)}^4 \leq K_m\|\varphi\|_{L^2(\R^3)}^{\frac{4}{5}} \|\varphi\|_{\dot{H}^1_m(\R^3)}^{\frac{6}{5}}\|\varphi\|_{L^5(\R^3)}^2
\end{equation}
is given by 
\begin{equation*}
K_{m} = \frac{7}{\sqrt[5]{3^35^2}}\(\rho_*\)^{-\frac{2}{5}}. 
\end{equation*}
Hence, 
\[
\rho_* = \sqrt{\frac{7^5}{3^35^2}}\,\(K_m\)^{-\frac{5}{2}},
\]
is uniquely determined from the best constant $K_m$ in \eqref{eq:GN3d}.
Therefore, we fix the mass $\rho>\rho_*$ such that $-\infty < e(\rho) < 0$ and select a minimizing sequence $(\varphi_n)_{n\ge 0} \subset H^1_m(\R^3)$ for \eqref{eq:Emin}.

\smallskip
{\bf Step 2:} We first prove that vanishing does not occur in the $z$-direction. Since $E_m(\varphi_n)\rightarrow e(\rho)<0$, we have that for $n$ sufficiently large, $E_m(\varphi_n)\le e(\rho)/2$. Therefore, $ \|\varphi_n\|_{L^4(\R^3)}^4\ge   |e(\rho)| > 0.$ We now decompose $\R^3$ as the disjoint union of $S_k=\R^2\times (k,k+1)$ (or $\R^+\times (k,k+1)$ when $x=(r,z)$) for $k\in\Z$, and write
\begin{align*}
0 < |e(\rho)| &\le \int_{\R^3} |\varphi_n|^4\,dx = \sum_{k} \int_{S_k} |\varphi_n|^4\,rdrdz\\
&\le \sum_{k} \left(\int_{S_k} |\varphi_n|^4\,rdrdz\right)^{\frac{1}{2}}\left(\int_{S_k} |\varphi_n|^4\,rdrdz\right)^{\frac{1}{2}}\\
&\le \sup_k \left(\int_{S_k} |\varphi_n|^4\,rdrdz\right) \sum_{k} \left(\int_{S_k} |\varphi_n|^4\,rdrdz\right)^{\frac{1}{2}}\\
&\lesssim \sup_k \left(\int_{S_k} |\varphi_n|^4\,rdrdz\right) \|\varphi_n\|_{H^1_m(\R^3)},
\end{align*}
where in the last step we have used the compactness lemma for bounded domain in variable $z$ (see, e.g., \cite[Lemma 2.1]{BBBS} or \cite[Lemma 5]{BV2003}) and summed up to recover the entire space. We next define 
\[
F_n(z):= \int_{\R^+}|\varphi_n(\cdot ,z)|^4\,rdr,\  \text{then}\ \sup_k\left(\int_{k}^{k+1} F_n\,dz\right) > 0, 
\]
and thus, no vanishing occurs.

\smallskip
{\bf Step 3:}  Next, if dichotomy occurs, for all $\alpha \in (0,\rho)$, we have
\begin{equation}
\label{eq:sub-add}
e(\rho) < e(\alpha) + e(\rho-\alpha).
\end{equation}
Indeed, using the scaling, $\varphi_\theta(x) = \varphi(\theta^{-1/3}x)=\varphi(\theta^{-1/3}r,\theta^{-1/3}z)$ with $\theta \in (1,\tfrac{\rho}{\alpha}]$, we have
\begin{align*}
e(\theta\alpha) &= \inf_{M(\varphi_\theta)=\theta\alpha} \left\{ \frac{1}{2}\| \varphi_\theta|\|^2_{\dot{H}^1_m} -\frac{1}{2}\|\varphi_\theta\|^4_{L^4(\R^3)} + \frac{2}{5} \|\varphi_\theta\|^5_{L^5(\R^3)}\right\}\\
&=\inf_{M(\varphi)=\alpha} \left\{ \frac{\theta^{\frac{1}{3}}}{2}\| \varphi|\|^2_{\dot{H}^1_m} -\frac{\theta}{2}\|\varphi\|^4_{L^4(\R^3)} + \frac{2\theta}{5} \|\varphi\|^5_{L^5(\R^3)}\right\}\\
&< \theta \inf_{M(\varphi)=\alpha} \left\{ \frac{1}{2}\| \varphi|\|^2_{\dot{H}^1_m} -\frac{1}{2}\|\varphi\|^4_{L^4(\R^3)} + \frac{2}{5} \|\varphi\|^5_{L^5(\R^3)}\right\} = \theta e(\alpha).
\end{align*}
Therefore, if $\alpha \geq \rho - \alpha$, we have
\[
e(\rho) < \frac{\rho}{\alpha}e(\alpha) = \left(1+\frac{\rho-\alpha}{\alpha}\right)e(\alpha) \le e(\alpha) + e(\rho-\alpha).
\]
 Let $\xi=(r,z)\in\R^2$ and choose cut-off function $\theta_k\in C^{\infty}(\R^2)$, such that 
\[
0\leq \theta_k\le 1,\ \ \ \theta_k =1\,\,\text{for}\,\, |\xi|\le R_k,\ \ \ \theta_k =0\,\,\text{for}\,\, |\xi|\ge  2R_k,\ \ \ |\nabla \theta_k|\le \frac{2}{R_k}.  
\]
Set 
\[
\varphi_{k,1}(r,z) = \varphi_{n_k}(r,z)\theta_k(r,z)\in H^1_m(\R^3)\ \text{ and }\ \varphi_{k,2} = \varphi_{n_k}(r,z)(1-\theta_k(r,z))\in H^1_m(\R^3).
\] 
Then, the sequences, $(\varphi_{k,1})_{k\ge 0}$, $(\varphi_{k,2})_{k\ge 0}$ satisfy
\begin{equation*}
  \operatorname{supp} \varphi_{k,1}\cap \operatorname{supp} \varphi_{k,2}=\emptyset,\,\qquad |\varphi_{k,1}| + |\varphi_{k,2}| \le |\varphi_{n_k}|,
\end{equation*}  
\begin{equation*}
  \|\varphi_{k,1}\|_{H^1_m(\R^3)} + \|\varphi_{k,2}\|_{H^1_m(\R^3)}\le C\|\varphi_{n_k}\|_{H^1_m(\R^3)}.
\end{equation*}  
Moreover,
\begin{equation*} 
  \|\varphi_{k,1}\|_{L^2(\R^3)}^2\Tend k \infty \alpha,\quad
    \|\varphi_{k,2}\|_{L^2(\R^3)}^2\Tend k \infty \rho-\alpha,
  \end{equation*}
 \begin{equation*}
 \left|\int_{\R^3} |\varphi_{n_k}|^p -\int_{\R^3}|\varphi_{k,1}|^p- \int_{\R^3}|\varphi_{k,2}|^p\right|\Tend k \infty0,
\end{equation*}
for all $2\le p<6$. We now compute
\begin{equation*}
\begin{aligned}
\|\varphi_{k,1}\|_{\dot{H}^1_m(\R^3)}^3 = \int_{\R}&\(\int_{\R^+} \(|\nabla (\varphi_{n_k}\theta_k)|^2 + m^2\frac{|\varphi_{n_k}\theta_k|^2}{|x|^2}\)rdr\)dz\\
= \int_{\R} &\(\int_{\R^+}\theta_k^2\(|\nabla\varphi_{n_k}|^2 + m^2\frac{|\varphi_{n_k}|^2}{|x|^2}\) rdr\)dz\\
&+ \int_{\R}\( \int_{\R^+} \big[|\varphi_{n_k}|^2|\nabla\theta_k|^2 + 2\RE\(\varphi_{n_k}\theta_k\nabla\bar{\varphi}_{n_k}\nabla\theta_k\)\big]rdr\)dz.
\end{aligned}
\end{equation*}
We then have
\begin{equation*}
\begin{aligned}
\int_{\R}&\(\int_{\R^+}\big[|\varphi_{n_k}|^2|\nabla\theta_k|^2 + 2\RE\(\varphi_{n_k}\theta_k\nabla\bar{\varphi}_{n_k}\nabla\theta_k\)\big]rdr\)dz\\
&\leq \frac{4}{R_k^2} \int_{\R}\(\int_{\R^+}|\varphi_{n_k}|^2rdr\)dz \\
&\qquad+ \frac{4}{R_k}\(\int_{\R}\(\int_{\R^+}|\varphi_{n_k}|^2rdr\)dz\)^{\frac{1}{2}}\(\int_{\R}\(\int_{\R^+} |\nabla \varphi_{n_k}| ^2rdr\)dz\)^{\frac{1}{2}}\rightarrow 0 
\end{aligned}
\end{equation*}
since $R_k\rightarrow +\infty$. Hence, we infer 
\begin{equation}\label{eq:kinetic1-3d}
\left|\|\varphi_{k,1}\|_{\dot{H}^1_m(\R^3)}^2 - \int_{\R}\( \int_{\R^+}\theta_k^2\(|\nabla\varphi_{n_k}|^2 + m^2\frac{|\varphi_{n_k}|^2}{|x|^2}\)rdr\)dz\right| \rightarrow 0.
\end{equation}
In a similar fashion, we have
\begin{equation}\label{eq:kinetic2-3d}
\left|\|\varphi_{k,2}\|_{\dot{H}^1_m(\R^3)}^2 - \int_{\R}\(\int_{\R^+} (1-\theta_k)^2\(|\nabla\varphi_{n_k}|^2 + m^2\frac{|\varphi_{n_k}|^2}{|x|^2}\)rdr\)dz\right| \rightarrow 0.
\end{equation}
From \eqref{eq:kinetic1-3d} and \eqref{eq:kinetic2-3d}, we obtain
\begin{equation*}
\begin{aligned}
\|\varphi_{n_k}\|_{\dot{H}^1_m(\R^3)}^2 &\ge \int_{\R}\(\int_{\R^+} \(\theta_k^2 + (1-\theta_k)^2\)\(|\nabla\varphi_{n_k}|^2 + m^2\frac{|\varphi_{n_k}|^2}{|x|^2}\)rdr\)dz\\
&\ge \|\varphi_{k,1}\|_{\dot{H}^1_m(\R^3)}^2 + \|\varphi_{k,2}\|_{\dot{H}^1_m(\R^3)}^2 + o(1).
\end{aligned}
\end{equation*}
We therefore conclude (by construction) that 
\begin{equation*}
  \liminf_{k\to \infty}\(E_m\(\varphi_{n_k}\)-E_m(\varphi_{k,1})-E(\varphi_{k,2})\)\ge 0,
\end{equation*}
which yields
\begin{align*}
e(\rho) \geq \limsup_{k\to \infty} \(E_m(\varphi_{k,1})+E_m(\varphi_{k,2})\)&\geq \liminf_{k\to \infty} \(E_m(\varphi_{k,1})+E_m(\varphi_{k,2})\)\\
&\geq e(\alpha) + e(\rho-\alpha).
\end{align*}
This contradicts \eqref{eq:sub-add}, thereby ruling out the dichotomy scenario.

  \smallskip
{\bf Step 4:} Hence, we conclude that compactness must occur. In particular, by \cite[Lemma 4]{BV2003}, which is based on the concentration-compactness lemma of \cite{PL284a} (see also \cite[Proposition 1.7.6(i)]{CazCourant}), we have that there exists a sequence $(r_k,z_k)\subset \R^2$, a minimizing subsequence $\varphi_{n_k}$ and a limiting object $\psi \in H^1_m(\R^+\times I)$ such that the mass concentrates on the strips $\R^+\times[-a,a]$ for sufficiently large $a$ (see \cite[Remark 9]{BV2003}). That is, the sequences $(r_k)$ and $(z_k)$ are bounded in the spirit of \cite[Lemma 1.7.4]{CazCourant}. Moreover, the compact embedding for cylindrical symmetry, \cite[Lemma 5]{BV2003} yields that
$\varphi_{n_k}\to \psi$ in $L^p_{\rm rad}(\R^2\times I)$ for all $2< p<6$ up to a (finite) translation. Therefore, using the compactness together with the lower-semicontinuity property of $H^1_m(\R^3)\simeq H^1_{\rm rad}(\R^+\times \R)$ norm, we get
\[
E_m\(\psi\) \leq \liminf_{k\rightarrow \infty} E_m\(\varphi_{n_k}\) = e\(\rho\),
\] which yields the existence of a minimizer for \eqref{eq:Emin}. By the properties of Steiner symmetrization (see, e.g., \cite[Section 4]{BroSo2000}) we deduce that a minimizer $\psi(|x|,z) = \psi(r,|z|)>0$ is non-increasing in $|z|$. 
\end{proof}

The fact that energy minimizers $\psi\in \mathcal{E}(\rho)$ are obtained via the concentration-compactness principle would in principle allow 
one to conclude an analogous orbital stability result as given in Proposition \ref{prop:stability}. 
We will not do so here, however, since this would require us to first establish a global well-posedness result for 
NLS in 3D with a {\it partial} inverse-square potential (recall that in our case the term $r^{-1}$ is always two-dimensional). 
To our knowledge no results (in particular no Strichartz-estimates) 
for such a model are currently available in the literature.

\begin{remark} We finally note that in the case of a nonlinear Klein-Gordon models, it is has been shown that central vortex states with $m\not =0$ cannot be achieved as 
global minimizers of the energy $E(u)$ with fixed mass and nonzero vorticity $L(u)\not =0$, see \cite[Section 2.2]{BBBS}. This is consistent with the results of \cite{GLY}.
\end{remark}


\bigskip

\bibliographystyle{siam}

\end{document}